\documentclass[12pt]{article}

\usepackage{latexsym,amsmath,amssymb}
\usepackage{ascmac} %%\UTF{0098}g\UTF{0082}\UTF{00C5}\UTF{0088}\UTF{00CD}\UTF{0082}\UTF{00DE} \itembox\UTF{0082}\UTF{00C6}\UTF{0082}\UTF{00A9}%
\usepackage{amscd}
\usepackage{amsfonts}
\usepackage{mathrsfs}
\usepackage{bm}  %% \UTF{008E}\UTF{00C0}\UTF{0090}\UTF{0094}\UTF{0082}\UTF{00E2}\UTF{0095}\UTF{00A1}\UTF{0091}f\UTF{0090}\UTF{0094}\UTF{0082}\UTF{00CC}\UTF{008F}W\UTF{008D}\UTF{0087}\UTF{0082}\UTF{00CC}\UTF{008B}L\UTF{008D}\UTF{0086}\UTF{0082}\UTF{00F0}\UTF{008A}\UTF{00C8}\UTF{0092}P\UTF{0082}\UTF{00C9}\UTF{0095}\\UTF{008E}\UTF{00A6}\UTF{0082}\UTF{00B7}\UTF{0082}\UTF{00E9}\UTF{0082}\UTF{00BD}\UTF{0082}\UTF{00DF}%
\usepackage{enumerate}  %% \UTF{0094}\UTF{00D4}\UTF{008D}\UTF{0086}\UTF{0095}t\UTF{0082}\UTF{00AB}\UTF{0089}\UTF{00D3}\UTF{008F}\UTF{00F0}\UTF{008F}\UTF{0091}\UTF{0082}\UTF{00AB}\UTF{0082}\UTF{00F0}\UTF{008A}\UTF{00C8}\UTF{0092}P\UTF{0082}\UTF{00C9}\UTF{0093}\UTF{00FC}\UTF{0097}\UTF{00CD}\UTF{0082}\UTF{00B7}\UTF{0082}\UTF{00E9}%
\usepackage{hhline} %% \UTF{0095}\\UTF{0091}g\UTF{0082}\UTF{00CC}\UTF{008C}r\UTF{0090}\UTF{00FC}\UTF{0082}\UTF{00AA}\UTF{0082}\UTF{00A9}\UTF{0082}\UTF{00C1}\UTF{0082}±\UTF{0082}\UTF{00E6}\UTF{0082}\UTF{00AD}\UTF{0082}\UTF{00C8}\UTF{0082}\UTF{00E9}%

\usepackage{amsthm}
\theoremstyle{definition}
\newtheorem{Theorem}{\bf Theorem}[section]
\newtheorem{Lemma}[Theorem]{\bf Lemma} 
\newtheorem{Proposition}[Theorem]{\bf Proposition} 
\newtheorem{Corollary}[Theorem]{\bf Corollary}
\newtheorem{Remark}[Theorem]{\bf Remark}
\newtheorem{Example}[Theorem]{\bf Example}
\newtheorem{Definition}[Theorem]{\bf Definition}

\newenvironment{thm}{\begin{Theorem}}{\end{Theorem}}
\newenvironment{lem}{\begin{Lemma}}{\end{Lemma}}
\newenvironment{prop}{\begin{Proposition}}{\end{Proposition}}

\newenvironment{rem}{\begin{Remark}}{\end{Remark}}

\newenvironment{dfn}{\begin{Definition}}{\end{Definition}}

\makeatletter
    
    \@addtoreset{equation}{section}
  \makeatother

\makeatletter
\def\tr{\mathop{\operator@font tr}\nolimits}  
\makeatother

\makeatletter
\def\dist{\mathop{\operator@font dist}\nolimits}  
\makeatother

\makeatletter
\def\div{\mathop{\operator@font div}\nolimits}  
\makeatother

\makeatletter
\def\exp{\mathop{\operator@font exp}\nolimits}  
\makeatother

\makeatletter
\def\essinf{\mathop{\operator@font {\it ess}.\inf}\nolimits}  
\makeatother

\makeatletter
\def\esssup{\mathop{\operator@font {\it ess.}\sup}\nolimits}  
\makeatother

\newcommand{\R}{\mathbb{R}}
\newcommand{\N}{\mathbb{N}}

\newcommand{\M}{{\cal M}}

\def\M{\mathcal{M}}
\def\A{\mathcal{A}}
\def\a{\alpha}

\def\phi{\varphi}
\def\e{\varepsilon}

\def\fr{\frac}

\def\le{\left}
\def\ri{\right}
\def\fr{\frac}

\def\ol{\overline}

\begin{document}

\title{ABP maximum principles for fully nonlinear integro-differential equations\\ with unbounded inhomogeneous terms}

\author{
Shuhei Kitano \footnote{e-mail: sk.koryo@moegi.waseda.jp}\\
\ \\
Department of Applied Physics,\\
Waseda University\\
Tokyo, 169-8555,\\
JAPAN
}
\date{}

\pagenumbering{roman}
\maketitle
%\tableofcontents
%\newpage
\pagenumbering{arabic}

\begin{abstract}
Aleksandrov-Bakelman-Pucci maximum principles are %is 
studied for a class of %the 
fully nonlinear integro-differential equations of order $\sigma\in [2-\e_0,2)$, where $\e_0$ is a small constant depending only on given parameters. The goal of this paper is to improve an estimate %the result 
of Guillen and Schwab (Arch. Ration. Mech. Anal., {\bf 206}, 2012) in order to avoid the dependence on  $L^\infty$ norm of %the  to the estimate  depending only on $L^n$ norm 
 the inhomogeneous term. 
\end{abstract}

\section{Introduction}%%%%%%%%%%%%%%%%%%%%%%SECTION#1
\label{sec:intro}

In this paper, we study the fully nonlinear nonlocal equation of the form:
\begin{equation}\label{super}
\left\{
\begin{split}
\M^-u(x)&%:=\inf_{\lambda\leq Tr(A)\mbox{ and }A\leq\Lambda Id}\left\{(2-\sigma)\int_{\R^n}\delta(u,x,y)\frac{y^TAy}{|y|^{n+\sigma+2}}dy\right\}\\&
\leq f(x)\quad\mbox{in }B_1,\\
u(x)&\geq0\quad\mbox{in }\R^n\setminus B_1,
\end{split}
\right.
\end{equation}
where 
by setting 
$$\delta(u,x,y):=u(x+y)+u(x-y)-2u(x),$$
we define the minimal fractional Pucci operator:
$$
\M^-u(x):=\inf_{\lambda\leq Tr(A)\mbox{ and }O\leq A\leq\Lambda Id}\left\{(2-\sigma)\int_{\R^n}\delta(u,x,y)\frac{y^TAy}{|y|^{n+\sigma+2}}dy\right\}.
$$
Throughout this paper, we suppose that
$$n\geq 2,\quad 0<\lambda\leq n\Lambda,$$
and define $B_r:=\{ y \in \R^n \ : \ |y|<r\}$.  
Furthermore,  $Q_r\subset\R^n$ denotes the open cube with its center at $0$ and its side-length $r$. 
%the radius $r>0$ and the side-length $r$>0$. 
We also  set $B_r(x):=x+B_r$ and $Q_r(x):=x+Q_r$. 

%%%%%%%%%%     NECESSARY ??
We will also use the maximal Pucci operator defined by 
$$
\M^+u(x):=\sup_{\lambda\leq Tr(A)\mbox{ and }O\leq A\leq\Lambda Id}\left\{(2-\sigma)\int_{\R^n}\delta(u,x,y)\frac{y^TAy}{|y|^{n+\sigma+2}}dy\right\}.
$$
%%%%%%%%%%%

The main purpose of this paper is to show the Aleksandrov-Bakelman-Pucci (ABP %
for short) maximum principle depending only on $L^n$ norm of $f^+:=\max\{ f,0\}$. 
\begin{thm}\label{main}
There exist %a 
constants $\hat C>0$ and $\e_0\in (0,1)$ depending only on $n,\lambda$ and $\Lambda$ such that if $u\in L^\infty(\R^n)\cap LSC(\R^n)$ is a viscosity supersolution of \eqref{super} with $\sigma\in [2-\e_0,2)$ and $f\in C(\overline{B_1})$, then it follows that
\begin{equation}\label{ABP}
-\inf_{B_1}u\leq \hat C\|f^+\|_{L^n(\{u\leq0\})}.
\end{equation}
\end{thm}

The ABP maximum principle plays a fundamental role in the regularity theory of fully nonlinear equations (see \cite{Caf, CafCab} for instance). 
The ABP maximum principle for the second order equation %%gives 
 presents a bound for the infimum of supersolutions by $L^n$ norm of inhomogeneous terms. 
In the case of  nonlocal equations, the ABP maximum principle is 
%%studied 
investigated 
in \cite{CS09,GS12,MS} and references therein. 
More precisely, 
for viscosity solutions, 
Caffarelli and Silvestre prove 
that the maximum is estimated 
%% given 
from above by Riemann sums 
of $f$ instead of their 
$L^n$ norms in \cite{CS09}. 
A more quantitative version of the ABP maximum principle is 
%%proved 
established 
by Guillen and Schwab in \cite{GS12}. For an application of results in \cite{GS12}, $W^{\sigma,\e}$-estimate is %%studied 
obtained by Yu \cite{Y}, which was inspired by earlier works by \cite{E,L}. 
The ABP maximum principle for 
strong solutions of 
second order  equations with 
integral terms 
%%nonlocal terms 
is obtained by Mou and \'Swi\c{e}ch 
%%\cite{MS} 
although the nonlocal part is regarded as a lower order term in \cite{MS}. 

We shall recall the estimate shown by Guillen and Schwab from \cite{GS12}: 
%% obtain the following estimate:
\begin{thm}\label{GS}
There exists a constant $C_0=C_0(n,\lambda)$ such that if $u\in L^\infty(\R^n)\cap LSC(\R^n)$ is a viscosity supersolution of \eqref{super} with $\sigma\in(0,2)$ and $f\in C(\overline{B_1})$, then it follows that
\begin{equation}\label{ABP0}
-\inf_{B_1}u\leq C_0\|f^+\|_{L^\infty(\{u=\Gamma_\sigma\})}^{(2-\sigma)/2}\|f^+\|_{L^n(\{u=\Gamma_\sigma\})}^{\sigma/2},
\end{equation}
where $\Gamma_\sigma$ is defined in Section 2.
\end{thm}

Here, we shall  
%%mention 
briefly explain our
proof of Theorem \ref{main}.
%% and  \ref{GS}. 
In \cite{GS12}, Guillen and Schwab 
%%consider 
introduce 
the fractional order envelope $\Gamma_\sigma$ and its Riesz potential $P$, which 
interprets the fractional order equation as a second order partial differential equation. 
Applying the ABP maximum principle for second order equations, they show that $P$ is bounded 
from below 
by $L^n$ norm of $f$ as follow:
\[
-\inf P\leq \overline{C}\|f\|_{L^n(\{u=\Gamma_\sigma\})}.
\]
A more precise statement of the above is presented  in Lemma \ref{PRE1}. In order to estimate $\inf\Gamma_\sigma$ by $\inf P$, they obtain the lower bound of 
\[
|\{\Gamma_\sigma\leq \inf\Gamma_\sigma/2\}\cap (B_r(x_0)\setminus B_{r/2}(x_0))|,
\]
which is 
%%shown
derived from 
%% by 
a simple property of the equation. 
Here, $x_0$ is 
a point 
such that $\inf\Gamma_\sigma=\Gamma_\sigma(x_0)<0$,  and $r\in(0,-\Gamma(x_0)/(2f(x_0)))$. Hence it follows from the estimate that 
\[-\inf \Gamma_\sigma\leq C\|f^+\|_\infty^{(2-\sigma)/2}(-\inf P)^{\sigma/2}.
\]
Theorem \ref{GS} follows from above two inequalities. 
%%     INSERT INDENTION

In our proof of  Theorem \ref{main}, we introduce a new iteration procedure, which is based on an  
argument to show the weak Harnack inequality in the regularity theory for 
fully nonlinear PDE originated by Caffarelli in \cite{Caf}, 
%%On the other hand, 
%%Theorem \ref{main} is obtained by 
%%the  an
%% iteration procedure. Instead of the above, we use a method, which is usually used to 
%%establish
%%obtain the weak Harnack 
%%estimate inequality, 
%%
%%In order to obtain the bound of
in order to obtain the bound of%%%%%%%    CENTERING
$$|\{\Gamma_\sigma\leq \inf\Gamma_\sigma/2\}\cap(Q_r(x_0)\setminus Q_{r/4^{1/n}}(x_0))|$$
from the ABP maximum principle. 
%%We will give the precise definition of cubes $Q_r(x_0)$ later. 
Using this bound, one can show a new ABP maximum principle from 
%%old 
the original 
one.  
We inductively obtain a sequence of constants in ABP maximum principles, which yields \eqref{ABP} 
%%is obtained 
by taking 
the limit of this sequence.

We 
notice 
that 
%%$\{u=\Gamma_\sigma\}\subset\{u\leq0\}$ and 
$\{u=\Gamma_\sigma\}$ is replaced by $\{u\leq0\}$ in \eqref{ABP} 
while $\{u=\Gamma_\sigma\}\subset\{u\leq0\}$. 
However \eqref{ABP} has enough information 
%%in order 
 to prove H\"older 
 continuity 
estimates.

\begin{rem}
Although we will prove Theorem \ref{main} under the assumption that $f\in C(\overline{B_1})$, it remains to hold if we assume $f\in C(B_1)\cap L^n(B_1)$. Indeed, this is observed as follows: Since we may suppose that $\inf_{B_1}u<0$, we can find $x_0\in B_1$ such that $\inf_{B_1}u=u(x_0)$. We choose small $\eta>0$ such that $x_0\in B_{1-\eta}$. Notice that $f\in  C(B_1)\cap L^n(B_1)\subset C(\overline{B_{1-\eta}})$. Applying Theorem \ref{main} to $u(x)-\inf_{\R^n\setminus B_{1-\eta}}u$, we have
\[
u(x_0)\leq -\inf_{\R^n\setminus B_{1-\eta}}u+\hat{C}(1-\eta)^{\sigma-1}\|f\|_{L^n(\{u\leq0\}\cap B_{1-\eta})}.
\]
Letting $\eta\to0$, we obtain \eqref{ABP}.
\end{rem}

%%%%%%%%%%%%%%%%%%%%%%%%%%%%%%%%%%%%%%%%%%%%%%%%%%%%%%%%%%%%%%%%%%%%   Section   2    %%%%%%%%%%%%%%%%%%%%%%%%%%%%

\section{Preliminaries}

Throughout this paper, we let   $|\cdot|$ be the Euclidean norm. 

For a measurable subset $A$ of $\R^n$, $|A|$ is its Lebesgue measure, %%%  kamma
 and $\chi_A$ is its indicator function. 
 We denote by $S^{n-1}$ the $n-1$ dimensional unit sphere. We write $u^-:=\max\{-u,0\}$. We say $Q\subset Q_1$ is a dyadic cube if there exists $m\in \N$ such that $Q$ is obtained by dividing $Q_1$ to $2^{nm}$ cubes, %%% kamma
  and $\tilde{Q}$ is the predecessor of $Q$ if $Q$ is one of $2^n$ cubes 
%%obtained 
 constructed 
  from dividing $\tilde{Q}$. 
  We recall a lemma 
  %%%of 
  named 
the Caldel\'on-Zygmund cube decomposition.

%%%%%% (see Lemma 4.2 in \cite{CafCab}).

%%%%%%%%%        CHANGED
\begin{lem}\label{CalZyg}(Lemma 4.2 in \cite{CafCab}) 
Let  $A\subset B\subset Q_1$ be %%%%mesurable 
measurable 
%%%%
sets and $0<\delta<1$ such that

\begin{tabular}{ll}
(a)& $|A|\leq\delta$,\\
(b)&if $Q$ is a dyadic cube such that $|A\cap Q|>\delta|Q|$, then $\tilde{Q}\subset B$, where\\
& $\tilde{Q}$ is the predecessor of $Q$.
\end{tabular}
Then, $|A|\leq\delta |B|$.
\end{lem}

We recall the definition of viscosity 
%%  solution. 
 solutions. 
 We say that $\phi$ touches $u$ from below at $x$ in a neighborhood $U$ whenever
\[
u(x)=\phi(x)\quad\mbox{and}\quad u(y)\geq \phi(y)\quad\mbox{for }y\in U.
\]
\begin{dfn}
$u\in LSC(\R^n)\cap L^\infty(\R^n)$ is a  viscosity supersolution of $\M^-u=f$ in $B_1$ if whenever $\phi$ touches $u$ from below at $x\in  B_1$ in some neighborhood 
$U$ for $\phi \in C^2(\ol{U})$,   
\[
v:=
\left\{
\begin{split}
&\phi\quad {\rm in}\ U\\
&u\quad {\rm in}\ \R^n\setminus U
\end{split}
\right.
\]
satisfies that $\M^-v(x)\leq f(x)$.
\end{dfn}
Hereafter we often use the following convention: 
%%%%%%%%%%    "//////"
``$u$ satisfies $\M^-u\leq f$" means ``$u$ is a viscosity supersolution of $\M^-u=f$".

We introduce %%borrow 
some 
%%% notetion 
notations 
from \cite{GS12}: for $\a\in (0,2)$,  %%%Let
\[
\A(\alpha):=\pi^{\alpha-n/2}\frac{\Gamma((n-\alpha)/2)}{\Gamma(\alpha/2)},  %.
\]
we note that %% Since 
$\A(\alpha)/\alpha$ converges to a positive constant as $\alpha\to0$. 
%%%%
Hence, we can find 
%%obtain the 
%%
 a constant $c_0=c_0(n)>0$ such that %%satisfying that for any $\sigma\in(0,2)$,
\begin{equation}\label{c0}
\left(\frac{\sqrt{n}}{2}\right)^{-n+2-\sigma}\cdot\frac{\A(2-\sigma)}{8(1-4^{-(2-\sigma)/n})}\geq c_0\quad (\forall \sigma \in (0,2)).
\end{equation}

Let $v$ be a bounded function satisfying 
%%%for $x\in \R^n$,
\[
\int_{\R^n}\frac{|\delta(v,x,y)|}{|y|^{n+\sigma}}dy<\infty %%%
\quad (x\in \R^n).
\]
We define its fractional order hessian by
\begin{equation*}
D^{\sigma}v(x):=\frac{(n+\sigma-2)(n+\sigma)}{2}\A(2-\sigma)\int_{\R^n}\frac{y\otimes y}{|y|^{n+\sigma+2}}\delta(v,x,y)dy
,  %%%%
\end{equation*}
which is a real symmetric $n\times n$ matrics valued function in $\R^n$. The first eigenvalue of $D^\sigma v$ is defined by 
\begin{equation*}
E_\sigma v(x):=\inf_{\tau\in S^{n-1}}\{(D^\sigma v(x)\tau)\cdot\tau\}.
\end{equation*}
We consider the fractional order envelope $\Gamma_\sigma$ of $u$, which is an analogue of the classical convex envelope, defined by
\begin{equation*}
\Gamma_{\sigma}(x):=\sup\{v(x):E_\sigma(v)\geq0\mbox{ in }B_3,\mbox{ and }v\leq -u^{-}\mbox{ in }\R^n\}.
\end{equation*}
We can see 
that $\Gamma_\sigma$ is 
%%as 
the unique viscosity solution of the obstacle problem:
\begin{equation*}
\left\{
\begin{split}
\min\{E_\sigma (\Gamma_\sigma),-u^--\Gamma_\sigma\}&=0\quad\mbox{in }B_3,\\
\Gamma_\sigma&=0\quad\mbox{in }\R^n\setminus B_3.
\end{split}
\right.
\end{equation*}

We also consider the Riesz potential $P$ of $\Gamma_\sigma$,
\[
P(x):=\A(2-\sigma)\int_{\R^n}\frac{\Gamma_\sigma(y)}{|x-y|^{n-(2-\sigma)}}
dy
%%%%%%     =\frac{A(n,2-\sigma)}{|x|^{n-(2-\sigma)}}*f
,\quad (x\in\R^n). 
\]
The following estimate is proved in 
%%%
(6.1) in \cite{GS12}. 
%%% (see Section 6 in \cite{GS12}):

\begin{lem}\label{PRE1}
There exist $\overline{C},R>1$ depending only on $n$ and $\lambda$ such that if $u$ satisfies \eqref{super} with $\sigma\in(0,2)$ and $f\in C(\overline{B_1})$, then it follows that
\begin{equation}\label{pre1}
-\inf_{B_R}P\leq \overline{C}\|f\|_{L^n(\{u=\Gamma_\sigma\})}.
\end{equation}
\end{lem} 

We 
%%%  use 
use 
a modification of a 
%%the 
 barrier function in  
%%$\eta$ 
%%%%shown by 
 Corollary 9.3 in \cite{CS09}. 
 Although the next proposition holds even when $\sigma\in (0,1]$, we shall restrict ourselves only for 
 $\sigma\in (1,2)$ to simplify our statements in the proceeding argument.

\begin{prop}\label{BR}
%%For any $\sigma_0\in(0,2)$, 
There exist $M_1, M_2$ and a $C^{1,1}$ function $\eta$ depending only on $n,\lambda$ and $\Lambda$ %%and $\sigma_0$ 
such that

\noindent(i) $supp\ \eta\subset B_{2\sqrt{n}}$,

\noindent(ii) $\eta\leq-2\mbox{ in }Q_3\mbox{ and }\|\eta\|_\infty\leq M_1$,
%%%\quad\mbox{and}

\noindent(iii) for every $\sigma\in(1,2)$,  we have $\M^+ 
\eta(x)\leq M_2\xi(x)$ %%%everywhere, 
 in $\R^n$, 
 %%%
 where $\xi$ is a continuous function with support inside $B_{1/4}$ and such that $0\leq\xi\leq1$.
\end{prop}

%%%%%%%%%%%%%%%%%%%%%%%%%%%%%%%%%%%%%%%%%%%%%%%%%%%%%%%%%%%%%%%%%%%%%%%%%%     Section  3        %%%%%%%%%%%%%%%%%%%%%%

\section{Proof of Theorem \ref{main}}
Set 
\begin{equation}\label{mu}
\mu:=\frac{1}{(64M_2\sqrt{n})^n},
\end{equation}
where $M_2$ is the constant in Proposition \ref{BR}. 
We  then freeze %%%%%choose 
small 
$\e_0\in (0,1)$ 
%%small satisfying 
such that for $\sigma\in(1,2)$,
\begin{align}
C_0M_2^{(2-\sigma)/2}&\leq
%%%%%%%%%%%%%   \e_0^{-\sigma/2}
\e_0^{-\sigma/(2n)}\label{e1},\\
(256\sqrt{n}\e_0^{-1/n})^{\e_0}(c_0^{-1}\overline{C})^{\sigma-1}M_1^{1+\mu^{-1}\log2}&\leq\e_0^{-1/n},\label{e3}\\
c_0^{-1}\overline{C}&\leq\e_0^{-1/n}, \label{e4}\\ 
%%\quad\mbox{and}\\
2^{\e_0}(c_0^{-1}\overline{C})^{\sigma/2}M_1^{\frac 12+\mu^{-1}\log2}\e_0^{\sigma/2n}&\leq2^{-1}\label{e5},
\end{align}
where 
$C_0$, $c_0$ and $\overline C$ are constants form %%  is as in 
Theorem \ref{GS}, %% and $c_0$ is as in 
\eqref{c0}  and 
%%$\overline{C}$ is as in
 Lemma \ref{PRE1}, respectively,  and 
 $M_1$ and $M_2$ are 
 those 
 %% as in 
 from Proposition \ref{BR}. 
 %%with 
 %%when $\sigma_0=1$.

We will see that Theorem \ref{main} 
%% is obtained 
holds true when 
%%with 
$$\hat{C}:=\e_0^{-1/n}$$ 
%%from 
in view of Lemma \ref{prop} below, by sending $i\to\infty$ in \eqref{ABPi}.

Throughout this paper, we remind that this $\e_0\in (0,1)$ satisfies \eqref{e1}-\eqref{e5} in the above.  

%%%%%%%%%%%%%%%%%%%%%%%%%%%%%%%%%%
%%%%%%%%   Lemma  3.  1.   %%%%%%%%%%%%%%%%%

\begin{lem}\label{prop}
For any $i=0,1,...$ and $u$ satisfying \eqref{super} with $\sigma\in [2-\e_0,2)$ and $f\in C(\overline{B_1})$, it follows that
\begin{equation}\label{ABPi}
-\inf_{B_1}u\leq C_i \|f^+\|_{L^\infty(\{u\leq0\})}^{(2-\sigma)/2}\|f^+\|_{L^n(\{u\leq0\})}^{\sigma/2}+\e_0^{-1/n}\|f^+\|_{L^n(\{u\leq0\})}
,
\end{equation}
where $C_0$ is 
%%as in 
from Theorem \ref{GS}%%
, and $C_{i+1}:=2^{-1}C_{i}=2^{-(i+1)}C_0$.
\end{lem}
In the case of $i=0$, \eqref{ABPi} follows
 directly 
 from  \eqref{ABP0}. 

Assuming \eqref{ABPi} for all viscosity supersolutions of \eqref{super} when $i=\tilde i-1$, 
we shall  show \eqref{ABPi} for those when $i=\tilde i$. 

%%   We shall show \eqref{ABPi} 
 %% for any $i$ from former one
%%by contradiction.  
%%% The problem is that in order to show \eqref{ABPi} for any supersolution $u$, we apply former one to another supersolution $v$. To make the discussion easy, it seems better to consider $i$ which contradicts Lemma \ref{prop}, instead of the induction. 
%%Let us consider 
%We set 
%\[
%\tilde{i}:=\inf\left\{i\in\N\ \ \le| \ 
%\begin{split}
%&\eqref{ABPi} \mbox{ does not hold with } i \mbox{ for some }%%
%\\
%&\mbox{%%satisfy 
 %viscosity supersolution } u\mbox{ of } \eqref{super}.
%satisfies the hypothesis in Lemma \ref{prop}}
%\end{split}\right.\right\}.\]
%%We will 
%%prove show $\tilde{i}=\infty$, which completes the proof. 
%Hence, 
%%On the contrary, let us assume 
%in what follows, assuming 
%\begin{equation}\label{Hypo_Contradiction}
%\tilde{i}<\infty,
%\end{equation} we shall obtain a contradiction. 
%% in what follow. 
%We notice that 
%%In this case, 
%\eqref{ABPi} holds when 
%%%%with 
%$\tilde{i}-1$ for any viscosity supersolution $u$ of \eqref{super}. 
%%as in Lemma \ref{prop}.

%%%%%%%%%%%%%%%%%%%%%%%%%%%%%%%%%%%%%%%%%%%%%%%%%%%%%%%%%%%%%%%%%   Lemma   3. 2.   %%%%%%%%%%%%%%%%%%%%%%%%%%%%%%%

First, we prove the following lemma:
%%%%%%%%%%%%          DRASTIC  CHANGE  !!!!
\begin{lem}\label{PE1}
For $h\in C(\overline{B_{2\sqrt{n}}})$, we assume 
\begin{equation}\label{h_infty}
\|h\|_{L^\infty(B_{2\sqrt{n}})}\leq(C_{\tilde{i}-1}\e_0^{\sigma/2n})^{-2/(2-\sigma)}=:L_{\tilde{i}-1},
\end{equation}
and
\begin{equation}\label{h_n}
\|h\|_{L^n(B_{2\sqrt{n}})}\leq\frac{\e_0^{1/n}}{64\sqrt{n}}=:M_3.
\end{equation}
Let  $v\in L^\infty (\R^n)\cap LSC(\R^n)$ be any nonnegative 
viscosity supersolution of 
$$
\M^-v(x)\leq h(x)\quad\mbox{in }B_{2\sqrt n}.
$$
If $v$ satisfies 
%(i) $M^-u(x)\leq f(x)$ in $B_{2\sqrt{n}}$ with $\sigma\in [2-\e_0,2)$ and $f\in C(\overline{B_1})$, (ii) $u\geq0$ in $\R^n$, (iii) 
\begin{equation}\label{v_Q3}
\inf_{Q_3}v\leq1,
\end{equation}
then 
it follows that
\begin{equation}\label{pe1}
|\{v\leq M_1\}\cap Q_1|> \mu\e_0,
\end{equation}
where $M_1$ is the constant %%and $\tilde{C}$ are the constants 
%%as 
in Proposition \ref{BR}, 
%%with when 
%%$\sigma_0=1$%%
and $\mu$ is defined by \eqref{mu}.
\end{lem}

%\begin{rem}Again, we notice that $\M^-$ contains $\e_0\in (0,1)$ selected in the beginning of this section. \end{rem}

\begin{proof}
Let $\eta$ be 
%%as 
the function 
in Proposition \ref{BR}. 
%%with 
%%when $\sigma_0=1$. 
%%We consider 
We observe that 
$w:=v+\eta$ 
%%which 
satisfies
\begin{equation*}
\left\{
\begin{split}
\M^-w(x)&\leq h(x)+M_2\xi(x)=:g(x)\quad\mbox{in }B_{2\sqrt{n}},\\
w(x)&\geq0\quad\mbox{in }\R^n\setminus B_{2\sqrt{n}}.
\end{split}
\right.
\end{equation*}
Since $\eta\leq-2$ in $Q_3$%%
, and \eqref{v_Q3},  %$\inf_{\mathcal{Q}_3}u\leq1$, 
 we obtain $-\inf_{B_{2\sqrt{n}}}w\geq1$.

From the definition of  $\tilde{i}$, we can apply \eqref{ABPi} with $i=\tilde{i}-1$ to $w(2\sqrt{n}x)$, %%that is
 to derive 
\begin{align*}
-\inf_{B_{2\sqrt{n}}}w
&\leq C_{\tilde{i}-1}(2\sqrt{n})^{\sigma/2}\|g\|_{L^\infty(\{v\leq0\})}^{(2-\sigma)/2}\|g\|_{L^n(\{v\leq0\})}^{\sigma/2}+\frac{(2\sqrt{n})^{\sigma-1}}{\e_0^{1/n}}\|g\|_{L^n(\{v\leq0\})}.
\end{align*}

Since 
%%Note that 
$C_{\tilde{i}-1}=2^{-\tilde{i}+1}C_0\leq C_0$, 
by hypotheses \eqref{h_infty} and \eqref{e1}, we have 
%%  . It follows that
\begin{align*}
C_{\tilde{i}-1}(2\sqrt{n})^{\sigma/2}\|g\|_{L^\infty}^{(2-\sigma)/2}
&\leq C_{\tilde{i}-1}(2\sqrt{n})^{\sigma/2}(\|f\|_{L^\infty}^ {(2-\sigma)/2}+M_2^{(2-\sigma)/2})\\
&\leq (2\sqrt{n})^{\sigma/2}(\e_0^{-\sigma/2n}+C_0M_2^ {(2-\sigma)/2})\\
&\leq 2(2\sqrt{n})^{\sigma/2}\e_0^{-\sigma/2n}. %%,
\end{align*}
%where we apply the assumption (iv) to the second inequality and \eqref{e1} to the final inequality. 
Hence, we have
\begin{align}\label{lem33}
1\leq 2\{(2\sqrt{n})\e_0^{-1/n}\|g\|_{L^n(\{v\leq0\})}\}^{\sigma/2}+(2\sqrt{n})^{\sigma-1}\e_0^{-1/n}\|g\|_{L^n(\{v\leq0\})}.
\end{align}

Now, we claim
\[
\|g\|_{L^n(\{v\leq0\})}> \frac{\e_0^{1/n}}{32\sqrt{n}}.
\]
Indeed, if not, we have a contradiction because it follows from \eqref{lem33} that
\begin{align*}
1\leq 2\left(\frac{2\sqrt{n}}{32\sqrt{n}}\right)^{\sigma/2}+\frac{(2\sqrt{n})^{\sigma-1}}{32\sqrt{n}}\leq\frac{9}{16}.
\end{align*}
Hence, since $0\leq\xi\leq1$, $supp\ \xi\subset \overline{B_{1/4}}$, 
by assumption \eqref{h_n}, we have
\begin{align*}
\frac{\e_0^{1/n}}{32\sqrt{n}}<\|g\|_{L^n(\{v\leq0\})}
&\leq\|f\|_{L^n(B_1)}+M_2\| \xi\|_{L^n(\{v\leq0\})}\\
&\leq \frac{\e_0^{1/n}}{64\sqrt{n}}+M_2|\{v\leq0\}\cap B_{1/4}|^{1/n}%%.
\end{align*}
%%We have used that 
%%and the assumption (v). 
Therefore
, 
 we obtain
\begin{align*}
\frac{\e_0}{(64M_2\sqrt{n})^n}< |\{v\leq0\}\cap B_{1/4}|\leq|\{u\leq M_1\}\cap Q_1|.
\end{align*}
By recalling that $\mu=(64M_2\sqrt{n})^{-n}$ is given in \eqref{mu}, 
 \eqref{pe1} is now proved.
\end{proof}
%%%%%%%%%%%%%%%%%%%%%%%%%%%%%%%%%%%%%%%%%%%%%%%%%%%%%%%%%%%%%%%%%%%%%    Lemma  3. 3.   %%%%%%%%%%%%%%%%%%%%%%%%%%%

Although the next
%%Next 
 lemma  is rather standard, we give a proof for the reader's convenience. 
 
\begin{lem}\label{PE2}
Let $v$ be as in Lemma \ref{PE1}. Then
\begin{equation}\label{pe2}
|\{v>M^k\}\cap Q_1|\leq(1-\mu\e_0)^k
\end{equation}
for $k=1,2,...$, where $M_1$ and $\mu$ are constants from Proposition \ref{BR} and \eqref{mu}, respectively. 
\end{lem}

\begin{proof}
For $k=1$, \eqref{pe2} is trivial due to \eqref{pe1}. 
Suppose that \eqref{pe2} holds for $k-1$. 
Let
\begin{equation*}
A:=\{v>M_1^k\}\cap Q_1,\quad B:=\{v>M_1^{k-1}\}\cap Q_1.
\end{equation*}
\eqref{pe2} will be proved if we show that
\begin{equation}\label{cube2}
|A|\leq(1-\mu\e_0)|B|
\end{equation}
and \eqref{cube2} will be deduced by appllying  Lemma \ref{CalZyg}. We need to check conditions there. Clearly $A\subset B\subset Q_1$ and $|A|\leq|\{v>M\}\cap Q_1|\leq(1-\mu\e_0)$. 
We now prove that if $Q$ is a dyadic cube such that
\begin{equation}\label{cube}
|A\cap Q|\geq(1-\mu\e_0)|Q|,
\end{equation}
then $\tilde{Q}\subset B$, where $\tilde{Q}$ is the predecessor of $Q$. 
%%%%%
If not, there exists
%%On the contrary, for 
 some $Q=Q_{1/2^j}(x^*)$ satisfying \eqref{cube}, and 
 %%suppose 
 $\tilde{Q}\nsubseteq B$. 
 We then find 
 %% and let 
 $\tilde{x}\in\tilde{Q}$ 
 %%be 
 such that $v(\tilde{x})\leq M_1^{k-1}$.

Setting  functions
\[
\tilde{v}(x):=\frac{v(x^*+2^{-j}x)}{M_1^{k-1}}\quad\mbox{and }\tilde{h}(x):=\frac{h(x^*+2^{-j}x)}{2^{\sigma j}M_1^{k-1}},
\]
we observe that $\tilde{v}$ satisfies $\M^-\tilde{v}\leq \tilde{h}$ in $B_{2\sqrt{n}}$. 
Notice $\tilde{v}\geq0$ by definition. 
Since $\tilde{x}\in Q_{3/2^j}(x^*)$, we have $\inf_{Q_{3}}\tilde{v}\leq1$. 
Moreover, we have 
\[
\|\tilde{h}\|_{L^\infty(B_{2\sqrt{n}})}\leq\frac{\|h\|_{L^\infty(B_{2\sqrt{n}})}}{2^{\sigma j}M_1^{k-1}}\leq \|h\|_{L^\infty(B_{2\sqrt{n}})}\leq L_{\tilde{i}-1}
\]
and
\[
\|\tilde{h}\|_{L^n(B_{2\sqrt{n}})}\leq\frac{2^j\|h\|_{L^n(B_{2\sqrt{n}})}}{2^{\sigma j}M_1^{k-1}}\leq \|h\|_{L^n(B_{2\sqrt{n}})}\leq M_3.
\]
Since $\tilde{v}$ 
%%%%is under 
satisfies the hypotheses %%% 
of 
%%%%  lemma \ref{pe1}. By 
Lemma \ref{PE1}, it follows that
\begin{equation*}
\mu\e_0<|\{\tilde{v}\leq M_1\}\cap Q_1|=2^{jn}|\{v\leq M_1^k\}\cap Q_{1/2^i}(x^*)|
\end{equation*}
Hence $|Q\setminus A|>\mu\e_0|Q|$, 
which contradicts \eqref{cube}.
\end{proof}

%%%%%%%%%%%%%%%%%%%%%%%%%%%%%%%%%%%%%%%%%%%%%%%%%%%%%%%%%%%%%%%%%%%   Proof of Lemma   3.  1    %%%%%%%%%%%%%%%%%%%%%%%%%%%%%

\noindent{\bf Proof of Lemma \ref{prop}} 
We will prove \eqref{ABPi} with $\tilde{i}$ for any viscosity supersolution $u$ of \eqref{super}. 
%%%  which satisfies the hypothesis of Lemma \ref{prop}. 
Let $\hat g\in C_0(\R^n)$ be an arbitrary function satisfying $\hat g\geq f^+\chi_{\{u\leq 0\}\cap B_1}\mbox{ in }\R^n$. 
Without loss of generality, we can assume that $u$ satisfies
\begin{equation*}
\M^-u\leq f^+\chi_{\{u\leq0\}\cap B_1} \leq \hat{g}\quad\mbox{in }\R^n,
\end{equation*}
because $-u^-$ satisfies the above, instead of $u$.  
We shall show
\[
-\inf_{B_1}u\leq C_{\tilde{i}}\|\hat g\|_\infty^{(2-\sigma)/2}\|\hat g\|_n^{\sigma/2}+\e_0^{-1/n}\|\hat g\|_n,
\]
where $\|\hat g\|_\infty:=\|\hat g\|_{L^\infty(\R^n)}$ and $\|\hat g\|_n:=\|\hat g\|_{L^n(\R^n)}$. 
By taking the infimum of $\hat g$ so that $f^+\chi_{\{u\leq 0\}}\leq \hat g$, this inequality 
implies 
%%%is way, we obtain 
\eqref{ABPi} 
when %%with 
$i=\tilde{i}$ for $u$. %%%% by taking the infimum of $g$.

Since we may suppose $\inf_{B_1}u<0$, we can find $x_0\in B_1$ such that
%%Let $x_0\in B_1$ be such that 
%%%
$$\inf_{B_1}u=u(x_0)<0\quad (x_0\in B_1).$$
%%%%%%%%%%%
For $r\in (0,1)$, setting 
$$
N_0:=\fr{r^\sigma}{L_{\tilde i-1}}\| \hat g\|_\infty +\fr{r^{\sigma-1}}{M_3}\| \hat g\|_n,
$$
where $L_{\tilde{i}-1}$ and $M_3$ are constants from Lemma \ref{PE1}, 
%%%as in 
we  apply Lemma \ref{PE2} to
\[
u_r(x):=\frac{u(x_0+rx)-u(x_0)}{N_0}\quad\mbox{and}\quad 
%%L_{\tilde{i}-1}^{-1}r^\sigma\|\hat g\|_\infty+N^{-1}r^{\sigma-1}\|g\|_n}\mbox{ and }
h_r(x):=\frac{r^\sigma \hat g(x_0+rx)}{N_0}
%%%L_{\tilde{i}-1}^{-1}r^\sigma\|\hat g\|_\infty+N^{-1}r^{\sigma-1}\|\hat g\|_n}
\]
to obtain that for $k=1,2,...$, 
\begin{equation}\label{3Eq1}%%%%align*}
\begin{array}{rcl}
(1-\mu\e_0)^kr^n
&\geq &r^n|\{u_r>M_1^k\}\cap Q_1|\\
&=&\left|\left\{u>u(x_0)+M_1^k N_0%%\left(\frac{r^{\sigma}\|\hat g\|_\infty}{L_{\tilde{i}-1}}+\frac{r^{\sigma-1}\|\hat g\|_n}{N}\right)
\right\}\cap Q_r(x_0)\right|.
\end{array}
\end{equation}%%%%%align*}

We choose a large $k_0\in\N$ such that
\begin{equation}\label{3Eq2}
(1-\mu\e_0)^{k_0}\leq \frac 12<(1-\mu \e_0)^{k_0-1}.
\end{equation}
Next, setting $r_0:=\min\{ s_1,s_2,1\}$, where 
$$%%\begin{equation}%%%%%align*}
%%k_0&:=\left[\frac{-\log2}{\log(1-\mu\e_0)}\right],\quad 
s_1:=\left(\frac{-u(x_0)M_3}{4M_1^{k_0}\|\hat g\|_n}\right)^{1/(\sigma-1)}\quad\mbox{and}\quad 
s_2:=\left(\frac{-u(x_0)L_{\tilde{i}-1}}{4M_1^{k_0}\|\hat g\|_\infty}\right)^{1/\sigma},
%%%\quad\mbox{and}\quad r_0:=\min\{s_1,s_2,1\},
$$%%%\end{equation}%%%$$%%%\end{align*}where $[p]:=\min\{q\in\mathbb{Z}:p\leq q\}$. Since 
we easily observe that for any $0<r\leq r_0$,
$$%%%\begin{align*}
M_1^{k_0}N_0%%\left(\frac{r^\sigma\|g\|_\infty}{L_{\tilde{i}-1}}+\frac{r^{\sigma-1}\|g\|_n}{N}\right)
\leq M_1^{k_0}\left(\frac{s_2^\sigma\|\hat g\|_\infty}{L_{\tilde{i}-1}}+\frac{s_1^{\sigma-1}\|\hat g\|_n}{M_3}\right)%%%\\&
=-\frac{u(x_0)}{2}.
$$%%%%%\end{align*}
%%%and $(1-\mu\e_0)^{k_0}\leq1/2$, 
Thus, \eqref{3Eq1} together with \eqref{3Eq2} yields
\begin{equation}
\frac{r^n}{2}\geq \le|\le\{ u>\frac{u(x_0)}{2}%%+M^{k_0}N_0
\ri\}\cap Q_r(x_0)\ri| \quad (0<r\leq r_0).
\end{equation}
Therefore, we have 
\begin{align*}
\left|\left\{u\leq \frac{u(x_0)}{2}\right\}\cap Q_r(x_0)\right|%%= r^n-\left|\left\{u> \frac{u(x_0)}{2}\right\}\cap Q_r(x_0)\right|
\geq \frac{r^n}{2}.
\end{align*}
Let $r_l:=4^{-l/n}r_0$. For each $l=0,1,...$, we have
\begin{align*}
\left|\left\{u\leq \frac{u(x_0)}{2}\right\}\cap (Q_{r_l}(x_0)\setminus Q_{r_{l+1}}(x_0))\right|\geq\frac{r_l^n}{2}-|Q_{l+1}|= \frac{r_l^n}{4}.
\end{align*}
Note that $\Gamma_\sigma\leq u$ in $\R^n$ and $\Gamma_\sigma(x_0)=u(x_0)$. 
Hence, setting 
\begin{equation*}
A_l:=\left\{\Gamma_\sigma\leq \frac{\Gamma_\sigma(x_0)}{2}\right\}\cap (Q_{r_l}(x_0)\setminus Q_{r_{l+1}}(x_0)),
\end{equation*}
we see that  $|A_l|\geq r_l^n/4$ for $l=0,1,...$. 
Noting $|y|\leq 2^{-1}\sqrt nr_l$ for $y\in Q_{r_l}$, we thus calculate as follows: 
%%%%% and it follows
\begin{align}
-P(x_0)
&=\A(2-\sigma)\int_{\R^n}-\Gamma_\sigma(x_0+y)|y|^{-n+2-\sigma}dy\nonumber\\
&\geq \A(2-\sigma)\sum_{l=0}^\infty\int_{\{x_0+y\in A_l\}}-\frac{1}{2}\Gamma_\sigma(x_0)|y|^{-n+2-\sigma}dy\nonumber\\
&\geq \A(2-\sigma)\left(-\frac{1}{2}\Gamma_\sigma(x_0)\right)\left(\frac{\sqrt n}{2}\right)^{-n+2-\sigma}\sum_{l=0}^\infty|A_l|r_l^{-n+2-\sigma}\nonumber\\
&\geq \left(\frac{\sqrt{n}}{2}\right)^{-n+2-\sigma}\cdot\frac{-\Gamma_\sigma(x_0)\A(2-\sigma)r_0^{2-\sigma}}{8}\sum_{l=0}^\infty4^{-(2-\sigma)l/n}\nonumber\\ 
&=\left(\frac{\sqrt{n}}{2}\right)^{-n+2-\sigma}\cdot\frac{-\Gamma_\sigma(x_0)\A(2-\sigma)r_0^{2-\sigma}}{8(1-4^{-(2-\sigma)/n})}\nonumber\\&
\geq -c_0\Gamma_\sigma(x_0)r_0^{2-\sigma},\label{prop1}
\end{align}
where $c_0$ is from %%as in 
\eqref{c0}. Using \eqref{pre1} and \eqref{prop1}, we obtain
\begin{align*}
-\inf_{B_1}u=-\Gamma_\sigma(x_0)
\leq -c_0^{-1}r_0^{-(2-\sigma)}P(x_0)
&\leq c_0^{-1}\overline{C}r_0^{-(2-\sigma)}\|f^+\|_{L^n(\{u=\Gamma_\sigma\})}\\
&\leq c%%C
_0^{-1}\overline{C}r_0^{-(2-\sigma)}\|\hat g\|_n.
\end{align*}

We shall consider three cases of $r_0>0$. 
If we suppose $r_0=1$, then \eqref{ABP} holds true for $\hat C=c_0^{-1}\overline{C}$. 
Thus, due to \eqref{e4}, we conclude the proof.

We next assume $r_0=s_1$. 
In this case, it follows that 
\begin{align*}
-\inf_{B_1}u
&\leq c_0^{-1}\overline{C}\left(\frac{4M_1^{k_0}\|\hat g\|_n}{-M_3\inf_{B_1} u}\right)^{(2-\sigma)(\sigma-1)^{-1}}\|\hat g\|_n,
\end{align*}
which implies
\[
-\inf_{B_1}u
\leq (c_0^{-1}\overline{C})^{\sigma-1}(4M_1^{k_0}M_3^{-1})^{2-\sigma}\|\hat g\|_n.
\]
Because of $\e_0\geq2-\sigma$ and \eqref{3Eq2}, we have
\begin{align*}
k_0(2-\sigma)
\leq k_0\e_0 \leq \e_0-\frac{\e_0 \log2}{\log(1-\mu\e_0)}
\leq \e_0+\frac{\log 2}{\mu},
%%  \right]\e_0\leq \frac{-2\e_0\log2}{\log(1-\mu\e_0)}
%%  \leq\frac{2\log2}{\mu}.
\end{align*}
where the last inequality is valid because $z\leq -\log (1-z)$ for $0\leq z<1$. 

Recalling that $M_3=(64\sqrt{n})^{-1}\e_0^{1/n}$, 
 by \eqref{e3}, we conclude that
\begin{equation}\label{prop2}
-\inf_{B_1}u\leq (256\sqrt{n}\e_0^{-1/n})^{\e_0}(c_0^{-1}\overline{C})^{\sigma-1}M_1^{1+\mu^{-1}\log2}\|\hat g\|_n\leq\e_0^{-1/n}\|\hat g\|_n.
\end{equation}

In the case of $r_0=s_2$, we have
\begin{equation*}
-\inf_{B_1}u
\leq c_0^{-1}\overline{C}\left(\frac{4M_1^{k_0}\|\hat g\|_\infty}{-L_{\tilde{i}-1}\inf_{B_1} u}\right)^{(2-\sigma)/\sigma}\|\hat g\|_n.
\end{equation*}
Hence
\[
-\inf_{B_1}u
\leq (c_0^{-1}\overline{C})^{\sigma/2}(4M_1^{k_0}L_{\tilde{i}-1}^{-1})^{(2-\sigma)/2}\|\hat g\|_\infty^{(2-\sigma)/2}\|\hat g\|_n^{\sigma/2}.
\]
Recall that $L_{\tilde{i}-1}=(C_{\tilde{i}-1}\e_0^{\sigma/2n})^{-2/(2-\sigma)}$. 
It follows from \eqref{e5} that
\begin{align}
-\inf_{B_1}u
&\leq 2^{\e_0}(c_0^{-1}\overline{C})^{\sigma/2}M^{\frac12+\mu^{-1}\log2}\e_0^{\sigma/2n}C_{\tilde{i}-1}\|\hat g\|_\infty^{(2-\sigma)/2}\|\hat g\|_n^{\sigma/2}\nonumber\\
&\leq 2^{-1}C_{\tilde{i}-1}\|\hat g\|_\infty^{(2-\sigma)/2}\|\hat g\|_n^{\sigma/2}\nonumber\\
&=C_{\tilde{i}} \|\hat g\|_\infty^{(2-\sigma)/2}\|\hat g\|_n^{\sigma/2}. \label{prop3}
\end{align}

In any case, from \eqref{prop2} and \eqref{prop3}, we have
\[
-\inf_{B_1}u\leq C_{\tilde{i}} \|\hat g\|_\infty^{(2-\sigma)/2}\|\hat g\|_n^{\sigma/2}+\e_0^{-1/n}\|\hat g\|_n.
\]
Taking the infimum of $\hat g$ satisfying $\hat g\geq f^+\chi_{\{u\leq0\}\cap B_1}$ in $\R^n$, \eqref{ABPi} is proved with $\tilde{i}$. %%, which contradicts the definition of $\tilde{i}$.
$\quad\Box$
\vspace{10pt}

%%%%%%%%%%%%%%%%%%%%%%%%%%%%%%%%%%%%%%%%%%%%%%%%%%%%%%%%%%%%%%%%%%%%%%%%%%%%%%%%%%%%%%%%%%%%%%%%%%%%%%%%%%%%%%%%%%%%%%%%%%%%%%%%%%%%%%%%%%%%%%%%

%\section{Extention to unbounded inhomogeneous term}

%In 

\noindent{\bf Acknowledgements:}  The author wishes to express his thanks to Prof. Shigeaki Koike for many stimulating conversations and careful reading of the first draft. 

%%%%%%%%%%%%%%%%%%%%%%%%%%%%%%%%%%%%%%%%%%%%%%%%%%%%%%%%%%%%%%%%%%%%%%%%%%%%%%%%%%%%%%%%%%%%%%%%%%%%%%%%%%%%%%%%%%%%%%%%%%%%%%%%%%%%%%%%%%%%%%%%

\end{document}